\documentclass{amsart}

\usepackage{amssymb}
\usepackage{graphicx,color}

\usepackage{mathrsfs}

\newcommand{\sgn}{\mathrm{sgn}}     
\newcommand{\hyp}{B_T}  
\newcommand{\len}{\mathcal{L}}   
\newcommand{\origin}{\mathbf{0}}  

\newtheorem{main}{Theorem}

\newtheorem{thm}{Theorem}[section]
\newtheorem{lemma}[thm]{Lemma}
\newtheorem{prop}[thm]{Proposition}
\newtheorem{cor}[thm]{Corollary}

\begin{document}

\title{A Poincar\'{e} ball model for taxicab hyperbolic geometry}
\author{Aaron Fish*, Dylan Helliwell}
\date{\today}
\thanks{*Aaron Fish was supported by the Seattle University Mathematics Early Research REU in 2017.}
\thanks{We are deeply grateful to our reviewer who provided us with advice and suggestions that resulted in a significant improvement of this paper into its present form.  Any remaining errors are our own.}

\begin{abstract}
Taxicab space is a modification of Euclidean space that uses an alternative notion of distance.  Similarly, the Poincar\'{e} ball is a model of hyperbolic geometry that consists of a subset of Euclidean space with an alternative notion of distance.  In this paper, we merge these two variations to create a taxicab version of the Poincar\'{e} ball.  We determine the isometry group for this new space and show that this space is hyperbolic in the sense of Gromov.
\end{abstract}

\maketitle

\section{Introduction}

Taxicab space in  $n$ dimensions is the set $\mathbb{R}^n$ with the taxicab distance $d_T$ defined as follows:  for two points $x = (x_1, \ldots, x_n)$ and $y = (y_1, \ldots y_n)$,
\[
d_T(x, y) = \sum_{i = 1}^n |x_i - y_i|.
\]
This distance function arises from the taxicab or $\ell^1$ norm
\[
\|v\|_T = \sum_{i = 1}^n |v_i|.
\]
This space was introduced by Hermann Minkowski in the late 19th century as an alternative to Euclidean geometry and has since enjoyed a fair amount of interest.  See \cite{Krause, Reynolds, BCFHMNSTV, FHS} for details about this space along with various constructions and objects that are developed within it.

The Poincar\'{e} ball is a model of hyperbolic geometry consisting of the set $B = \left\{x \in \mathbb{R}^n: \sum_{i = 1}^n x_i^2 < 1\right\}$ equipped with the Riemannian metric
\[
g_P(v, w) = 4 \frac{v \cdot w}{(1 - \|x\|_E^2)^2}
\]
where $v$ and $w$ are tangent vectors at the point $x \in D$ and $\|\cdot\|_E$ is the Euclidean norm.  This gives rise to a norm on each tangent space 
\[
\|v\|_{P} = 2 \frac{\|v\|_E}{1 - \|x\|_E^2}
\]
and from this, the length of a piecewise smooth curve $\gamma:[0, 1] \rightarrow B$ is given by
\[
\len_P (\gamma) = \int_0^1 \| \gamma'(t) \|_P \, dt = \int_0^1 2 \frac{\|\gamma'(t)\|_E}{1 - \|\gamma(t)\|_E^2}\, dt.
\]
See for example \cite{Anderson} for more detail about hyperbolic geometry, including the Poincar\'{e} ball as a model.

In this paper, we create a taxicab Poincar\'{e} ball by following this development using taxicab norm
wherever possible.  Specifically, we define the taxicab Poincar\'{e} ball to be the set
\[
\hyp^n = \{x \in \mathbb{R}^n: \|x\|_T < 1\} = \left\{x \in \mathbb{R}^n: \sum_{i = 1}^n |x_i| < 1\right\}.
\]
This set together with its boundary is a regular $n$-orthoplex.  For example, when $n = 2$, $\hyp^2$ is an open square with its vertices on the coordinate axes and when $n = 3$, $\hyp^3$ is an open octahedron.  When the specific dimension, which will always be at least two, is not playing a critical role, the superscript $n$ may be omitted.

We equip this set with the norm on each tangent space given by
\[
\|v\|_{\hyp} = \frac{\|v\|_T}{1 - \|x\|_T^2}
\]
where $v$ is a vector based at $x \in \hyp$ and the coordinate basis is used to define the norm on the tangent space.  We start with a norm instead of a Riemannian metric because the taxicab norm does not arise from an inner product.  We also omit a factor of 2.  In the case of the standard Poincar\'{e} ball, this factor ensures that the resulting curvature is uniformly equal to $-1$, and we will not be attempting to compute curvature in $\hyp$.  That being said, see the comment after Theorem~\ref{gromovhyperbolicthm} below.

This norm allows us to measure the length of curves that admit absolutely continuous parameterizations.  Given two points $p$ and $q$, the set of such curves from $p$ to $q$ is denoted $\Gamma(p, q)$.  To determine the length minimizers among these curves, we introduce a point $m(p, q)$ which we call the minimal point because it turns out to be the point on the distance minimizer that is closest to the origin.  This point gives rise to an L-shapbed curve $\lambda_{p, q}$, which is just the concatenation of the Euclidean line segments from $p$ to $m(p, q)$ and from $m(p, q)$ to $q$.  Minimal points and L-shaped curves are defined in Section~\ref{minimalpointsubsec} and~\ref{curvesubsec} and illustrated in the case $n = 2$ in Figure~\ref{msandlambdasfig}.  With this, we have:

\begin{main} \label{minimizerthm}
For $p, q \in \hyp$ and $\gamma \in \Gamma(p, q)$, $\len(\gamma) \geq \len(\lambda_{p, q})$ with equality if and only if $\gamma$ is fully monotonic and passes through $m(p, q)$.
\end{main}

Here, $\len$ is the length functional for $\hyp$ and ``fully monotonic'' means that each component of the parameterization of $\gamma$ is monotonic.  By way of comparison, geodesics in taxicab space are fully monotonic curves.  On the other hand, geodesics in the Poincar\'{e} ball $B$ are arcs of circles perpendicular to $\partial B$, and these curves are used to define the distance between points $p, q \in B$, resulting in
\[
d_{P}(p, q) = \cosh^{-1}\left(1 + 2 \frac{\|p - q\|_E^2}{(1-\|p\|_E^2)(1-\|q\|_E^2)}\right).
\]
See \cite{Anderson} for the development of length and distance in the Poincar\'{e} ball.  Similarly, in $\hyp$, the curves that minimize the length functional $\len$ are given by Theorem~\ref{minimizerthm} and these curves are used to define a distance function on $\hyp$:

\begin{main} \label{distthm}
The distance function on $\hyp$ arising from the length functional $\len$ is
\begin{equation*}  
d(p, q) = \tanh^{-1}\bigl(\|p\|_T \bigr) + \tanh^{-1}\bigl( \|q\|_T \bigr)
		- 2 \tanh^{-1} \bigl( \|m\|_T \bigr),
\end{equation*}
where $m = m(p, q)$.
\end{main}

With the distance function established, we can determine the isometry group.  The isometry group for taxicab space $(\mathbb{R}^n, d_T)$ is isomorphic to $\mathbb{R}^n \rtimes H_n$ where $H_n$ is the hyperoctahedral group of rank $n$. The transformations associated to $\mathbb{R}^n$ are translations and the transformations associated to $H_n$ are permutations of the coordinates, reflections across hyperplanes containing the origin, and their compositions \cite{Schattschneider, KocOz}.  Meanwhile, the isometry group for the Poincar\'{e} Ball is isomorphic to the orthochronous Lorentz group $O^+(n, 1)$.  In these two cases, the isometry group acts transitively on the space, reflecting the fact that both taxicab space and usual hyperbolic space are homogeneous.  We find that the isometry group for $\hyp$ is significantly more restrictive.

\begin{main} \label{isometrythm}
The isometry group for $\hyp^n$ is isomorphic to $H_n$.
\end{main}

The fact that $H_n$ is isomorphic to a subgroup of the isometry group follows from the fact that permutations of the coordinates and reflections across coordinate hyperplanes containing the origin preserve the norm $\|\cdot \|_{\hyp}$ and hence must be isometries.  These isometries will be used regularly and without explicit mention to simplify various arguments throughout this paper.

Since the isometry group does not act transitively, $\hyp$ is not homogeneous.  Despite this, hyperbolicity can still be explored.  For a given pair of points $x$ and $y$, and base point $z$, let $G(x, y; z)$ be the Gromov product

\[
G(x, y; z) = \frac{1}{2} \left[d(x, z) + d(y, z) - d(x, y) \right].
\]
Introduced by Gromov (see for example \cite{Gromov}), a metric space is said to be $\delta$-hyperbolic if there exists a  $\delta \geq 0$ such that for all $x, y, z, w$, the following inequality holds:
\begin{equation} \label{gromovineq}
G(x, y; w) \geq \min \left\{ G(x, z; w), G(y, z; w) \right\} - \delta.
\end{equation}
For example, the usual hyperbolic space with sectional curvature -1, is $\ln(2)$-hyperbolic, and, as shown in \cite{NicaSpakula}, this value for $\delta$ is optimal.  Also, $\mathbb{R}$-trees are 0-hyperbolic.  As the following theorem shows, $B_T$ provides a new concrete example in this area, the proof of which relies on the fact that Theorem~\ref{distthm} allows us to explicitly calculate certain key Gromov products.

\begin{main} \label{gromovhyperbolicthm}
$\hyp$ is $\ln(3)$-hyperbolic and this value of $\delta$ is optimal.
\end{main}

As mentioned above, the norm used for $\hyp$ omits a factor of 2 that appears in the norm for the Poincar\'{e} ball.  If we were to reintroduce this factor of 2, distances between pairs of points $p$ and $q$ on a coordinate axis would be the same in both $\hyp$ and the usual Poincar\'{e}  ball.  This factor of 2 would also double the $\delta$ in Theorem~\ref{gromovhyperbolicthm}, so in this sense, $\hyp$ is somewhat less hyperbolic than usual hyperbolic space.

Theorems~\ref{minimizerthm} and \ref{distthm} show that $\hyp$ is a geodesic space, though for most pairs of points, geodesics are not unique.  Since $\hyp$ is not uniquely geodesic, it cannot be a CAT(0) space.

Additionally, as Theorem~\ref{mediancharacterizationthm} and Corollary~\ref{medspacecor} show, $\hyp$ is not a median space.  We show in Theorem~\ref{mediancharacterizationthm} that a given triple of points supports a median if and only if at least two of the three minimal points defined by this triple coincide, and from this, examples of triples of points that do or do not support medians can easily be generated.  A notable special case is that the minimal point $m(p, q)$ is the median for the set $\{p, q, \origin\}$, where throughout the paper, $\origin$ is the origin.

Despite the fact that $\hyp$ is not a median space, the Gromov hyperbolicity ensured by Theorem~\ref{gromovhyperbolicthm} implies that $\hyp$ is a coarse median space, as shown in \cite{Bowditchcoarse}.  Please see \cite{Bowditch} and the references therein for more on the connections between such spaces.

This paper is organized as follows:  In Section~\ref{preliminarysec} we introduce some preliminary constructions and results that support the remainder of the paper.  In Section~\ref{lengthsec} we prove Theorem~\ref{minimizerthm}, and with this, in Section~\ref{distanceisomsec} we prove Theorems~\ref{distthm}, and \ref{isometrythm}.  In Section~\ref{hyperbolicitysec} we explore the Gromov hyperbolicity of $\hyp$, proving Theorem~\ref{gromovhyperbolicthm}, and we find a characterization for which triples of points support a median.  Finally, we share some concluding remarks in Section~\ref{finalsec}.

\section{Preliminaries} \label{preliminarysec}

In this section, we introduce some terminology and discuss the analytical tools used to prove our results.

\subsection{Orthants, orthotopes, orthoplexes}

An $n$-orthant is a set of all points with the property that for each $i \in \{1, \ldots, n\}$, the $i^\mathrm{th}$ coordinates of all points share the same sign or are zero.   This generalizes quadrants in two dimensions and octants in three dimensions.  If no dimension is provided, the top dimension is implied.  Otherwise, the notation $k$-orthant will be used to indicate that for all but $k$ specific indices, the coordinates are zero for each point.

An orthotope is the higher dimensional generalization of a rectangle in two dimensions and a rectangular cuboid in three dimensions.  In this paper, only orthotopes with boundary faces lying in coordinate hyperplanes will be used, and we identify such an orthotope with a pair of vertices:
\[
R_{p, q} = \bigl\{x: \min\{p_i, q_i\} \leq x_i \leq \max\{p_i, q_i\} \bigr\}.
\]
This notation allows for the orthotope in question to be lower dimensional if some coordinates of $p$ and $q$ coincide.  In this instance, if $p$ and $q$ differ in $k$ coordinates, $R_{p, q}$ will be $k$-dimensional, and if this dimension is important, we will refer to $R_{p, q}$ as a $k$-orthotope.

An orthoplex is the higher dimensional generalization of an octahedron in three dimensions.  In this paper, only orthoplexes centered at the origin and with vertices lying on the coordinate axes and equidistant from the origin will be used.  Taxicab hyperbolic space itself is the interior of an orthoplex.  The symmetry group of an $n$ dimensional orthoplex is $H_n$.

\subsection{Minimal points} \label{minimalpointsubsec}

Given two real numbers $x$ and $y$, let
\[
\ell(x, y) =
\begin{cases}
\sgn(x) \min \{|x|, |y| \} & \mbox{if} \ \sgn(x) =  \sgn(y) \\
0 & \mbox{if} \ \sgn(x) \neq \sgn(y),
\end{cases}
\]
where $\sgn(x)$ is the sign of $x$:
\[
\sgn(x) =
\begin{cases}
1 & \mbox{if} \ x > 0 \\
-1 & \mbox{if} \ x < 0 \\
0 & \mbox{if} \ x = 0.
\end{cases}
\]
With this, given $p, q \in \hyp$ let $m(p, q) = \bigl(m_1(p, q), \ldots, m_n(p, q)\bigr)$ where
\[
m_i(p, q) = \ell(p_i, q_i).
\]
We call this point the minimal point associated to $p$ and $q$.
See Figure~\ref{msandlambdasfig} for examples of minimal points when $n = 2$.

\begin{figure}
\begin{picture}(260,260)
\put(0,10){
\includegraphics[scale = 1, clip = true, draft = false]{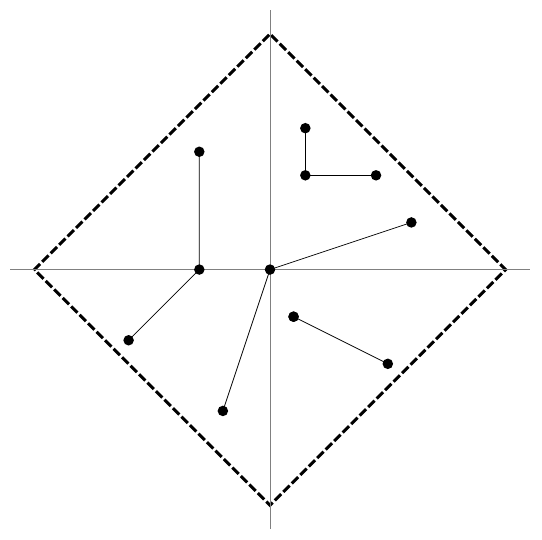}
}
\put(143,210){\scriptsize $p$}
\put(186,178){\scriptsize $q$}
\put(142,177){\scriptsize $m(p, q)$}
\put(103,200){\scriptsize $r$}
\put(69,101){\scriptsize $s$}
\put(72,144){\scriptsize $m(r, s)$}
\put(203,156){\scriptsize $u$}
\put(114,66){\scriptsize $v$}
\put(106,146){\scriptsize $m(u, v) = \origin$}
\put(140,122){\scriptsize $x = m(x, y)$}
\put(192,100){\scriptsize $y$}
\end{picture}
\caption{Some points in $\hyp^2$, the minimal points associated to various pairs of points, and the corresponding L-shaped curves. } \label{msandlambdasfig}
\end{figure}

Given two points $p$ and $q$ in the same orthant, we say $p$ lies beyond $q$ if $m(p, q) = q$.
We will find that one way to think about the minimal point for $p$ and $q$ is that it is the point furthest from the origin with the property that both $p$ and $q$ lie beyond it.

\subsection{Curves} \label{curvesubsec}

Unless otherwise stated, the domain for curves will be $[0, 1]$ and the codomain will be $\hyp$.  When we say ``let $\gamma$ be a curve from $p$ to $q$,'' we mean all components of $\gamma$ are absolutely continuous, $\gamma(0) = p$, and $\gamma(1) = q$.  The set of such curves is denoted $\Gamma(p,q)$.  If all components of a curve $\gamma$ are monotonic, we say $\gamma$ is fully monotonic.

Given two curves $\gamma$ and
$\varphi$ defined on $[a, b]$
with $\gamma(b) = \varphi(a)$, we define the concatenation $\gamma \ast \varphi$ as follows:
\begin{align*}
&\gamma \ast \varphi: [a, b] \longrightarrow \hyp \\
&\gamma \ast \varphi(t) =
\begin{cases}
\gamma(2t - a) \ &\mbox{if} \ a \leq t \leq \frac{a + b}{2} \\
\varphi(2t - b) \ &\mbox{if} \ \frac{a + b}{2}  < t \leq b.
\end{cases}
\end{align*}
We then extend this definition to curves with different domains by first applying a linear reparameterization to the second curve.

Given two points $p$ and $q$, we define the segment from $p$ to $q$ to be
\begin{align*}
&\sigma_{p, q}:[0, 1] \longrightarrow \hyp \\
&\sigma_{p, q}(t) = p+t (q - p).
\end{align*}
Given two points $p$ and $q$ and $m = m(p, q)$, we say the L-shaped curve from $p$ to $q$ is the curve
\begin{align*}
\lambda_{p, q}: [0, 1] \longrightarrow \hyp \\
\lambda_{p, q} = \sigma_{p, m} \ast \sigma_{m, q}.
\end{align*}
See Figure~\ref{msandlambdasfig} for examples of these curves when $n = 2$.

We define the length of a curve $\gamma:[a, b] \rightarrow \hyp$ as follows:
\[
\len(\gamma) = \int_a^b \|\gamma'(t)\|_{\hyp} \, dt
		= \int_a^b \frac{ \|\gamma'(t)\|_T }{1 - \|\gamma(t)\|_T^2}\, dt.
\]
Note that $\len(\gamma \ast \varphi) = \len(\gamma) + \len(\varphi)$.

Following \cite{BBI}, we consider the equivalence class of curves generated by declaring two curves $\gamma:[a_1,b_1] \rightarrow \hyp$ and $\eta:[a_2, b_2] \rightarrow \hyp$ to be equivalent if there exists an absolutely continuous monotonic change of variables $f:[a_1, b_1] \rightarrow [a_2, b_2]$ such that $\gamma = \eta \circ f$, and then extending to the finest equivalence class supporting these equivalencies.  We use $\sim$ to indicate this equivalence.  Some immediate consequences of this equivalence follow:  First, if two curves are equivalent, then they have the same length.  Second, through this equivalence, the specific domains and exact parameterizations for various curves are usually not important.  Third,  concatenation is associative: $\gamma \ast (\varphi \ast \eta) \sim (\gamma \ast \varphi) \ast \eta$.  Fourth, if $p$ lies beyond $q$ then, since $m(p, q) = q$, $\lambda_{p, q} =\sigma_{p, q} \ast \sigma_{q, q}$, so $\lambda_{p, q} \sim \sigma_{p, q}$.

We define the minimal point for a curve $\gamma$, $m(\gamma)$ analogously to the minimal point for two given points.  For a continuous function $f:[a, b] \rightarrow \mathbb{R}$, let
\[
\ell(f) = \sgn(f(a)) \min_{t \in [a, b]} \bigl| f(t) \bigr|.
\]
Note that the sign can be determined by any point in $[a, b]$ since, if the sign changes anywhere there must be a value where $f(t) = 0$ and so $\ell(f) = 0$.  With this, let $m(\gamma) = \bigl(m_1(\gamma), \ldots, m_n(\gamma)\bigr)$ where
\[
m_i(\gamma) = \ell(\gamma_i).
\]
Note that $m(\lambda_{p, q}) = m(p, q)$.

\subsection{Absolutely Continuous Functions}

As mentioned above, the curves under consideration will be parameterized using absolutely continuous functions.  We work at this level of regularity for a number of reasons:  First, absolutely continuous functions are differentiable almost everywhere, so the formula for the length of a curve makes sense.  Second, concatenation preserves absolute continuity.  Third, absolute continuity allows for substitution (change of variables) as an integration technique.  Fourth, if a function $f$ is absolutely continuous, then so is $|f|$.

In addition to these standard facts, absolute continuity is preserved under the following minimization processes, which are used to prove Theorem~\ref{minimizerthm}:  Given a continuous function $f:[a, b] \rightarrow \mathbb{R}$, we define the cumulative minimum function $\underline{f}:[a, b] \rightarrow \mathbb{R}$ and residual minimum function $\overline{f}:[a, b] \rightarrow \mathbb{R}$ as follows:
\[
\underline{f}(x) = \sgn\Bigl(f(a)\Bigr) \min_{t \in [a, x]} \, \bigl|f(t) \bigr|
\]
and
\[
\overline{f}(x) = \sgn\Bigl(f(b)\Bigr) \min_{t \in [x, b]} \, \bigl|f(t) \bigr|.
\]
See Figure~\ref{cumulativeandresidualminfig} for examples of such functions.

\begin{figure}
\begin{picture}(260,280)
\put(0,160){
\includegraphics[scale = 1, clip = true, draft = false]{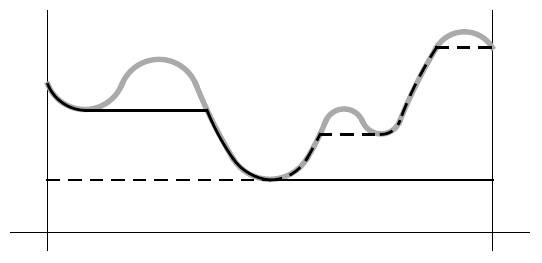}
}
\put(25,157){$a$}
\put(238,155){$b$}
\put(70,263){$f$}
\put(75,223){$\underline{f}$}
\put(80,188){$\overline{f}$}
\put(0,-10){
\includegraphics[scale = 1, clip = true, draft = false]{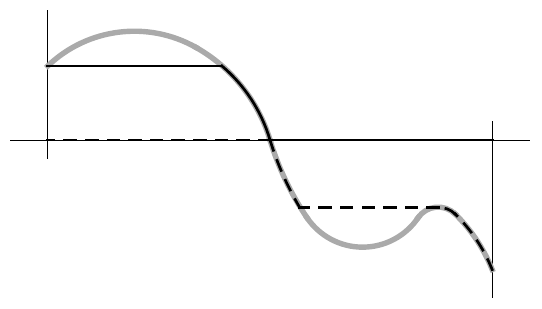}
}
\put(25,54){$a$}
\put(238,85){$b$}
\put(55,130){$f$}
\put(80,96){$\underline{f}$}
\put(170,45){$\overline{f}$}

\end{picture}
\caption{Absolutely continuous functions $f$ in gray, their cumulative minimum functions $\underline{f}$ in solid black, and their residual minimum functions $\overline{f}$ in dashed black.  In the top image, $f$ is positive.  In the bottom image, $f$ changes sign.} \label{cumulativeandresidualminfig}
\end{figure}

The following proposition collects the various technical facts we need about these functions.  The proofs are left to the reader.

\begin{prop} \label{minfunctionprop}
Let $f:[a, b] \rightarrow \mathbb{R}$ be absolutely continuous.
Then the cumulative minimum function $\underline{f}$ and residual minimum function $\overline{f}$ have the following properties:
\begin{itemize}
\item Both $\underline{f}$ and $\overline{f}$ are also absolutely continuous;
\item $\underline{f}$ and $\overline{f}$ are monotonic;
\item $\underline{f}(a) = f(a)$ and $\overline{f}(b) = f(b)$;
\item $\underline{f}(b) = \overline{f}(a)$;
\item for all $x \in [a, b]$, $\left| \underline{f}(x) \right| \leq |f(x)|$ and $\left| \overline{f}(x) \right| \leq |f(x)|$;
\item for all $x \in [a, b]$, $\left| \left( \underline{f} \right)'(x) \right| \leq \left| f'(x) \right|$ whenever both are defined and $\left| \left( \overline{f} \right)'(x) \right| \leq \left|f'(x) \right|$ whenever both are defined.
\item for all $x \in [a, b]$, if both $\left( \underline{f} \right)'(x)$ and $\left( \overline{f} \right)'(x)$ are defined, then one of them must be equal to zero.
\end{itemize}
\end{prop}

\section{Length minimizers} \label{lengthsec}

In this section, we determine the length minimizers for $\hyp$:

\medskip
\noindent \textbf{Theorem~\ref{minimizerthm}.} \emph{
For $p, q \in \hyp$ and $\gamma \in \Gamma(p, q)$, $\len(\gamma) \geq \len(\lambda_{p, q})$ with equality if and only if $\gamma$ is fully monotonic and passes through $m(p, q)$.
}
\medskip

Before proving Theorem~\ref{minimizerthm}, we establish a special case as a lemma.  This will be used in the proof and also plays a central role in establishing an explicit formula for length minimizers.

\begin{lemma} \label{dmqbeyondplemma}
Let $q$ lie beyond $p$ and let $\gamma \in \Gamma(p, q)$ be fully monotonic.  Then
\[
\len(\gamma) = \tanh^{-1} \bigl( \|q\|_T \bigr) - \tanh^{-1} \bigl( \|p\|_T \bigr).
\]
\end{lemma}

\begin{proof}
Define the new curve $\check{\gamma}$ with components $\check{\gamma}_i = | \gamma_i |$.
Note that this is a curve from
	$\bigl(|p_1|, |p_2|, \ldots |p_n| \bigr)$ to $\bigl(|q_1|, |q_2|, \ldots, |q_ n| \bigr)$ and that $\len(\check{\gamma}) = \len(\gamma)$.
Then
\begin{align*}
\len(\gamma) &= \len(\check{\gamma}) \\
			&= \int_0^1 \frac{\left| \left| \check{\gamma}'(t) \right| \right|_T}
				{1 - \left| \left| \check{\gamma}(t) \right| \right|_T^2}\, dt \\
			&= \int_0^1 \frac{ \sum_{i = 1}^n \check{\gamma}_i'(t)}
				{1 - \bigl( \sum_{i = 1}^n \check{\gamma}_i(t) \bigr)^2}\, dt \\
			&= \int_{\sum_{i = 1}^n \check{\gamma}_i(0) }
					^{\sum_{i = 1}^n \check{\gamma}_i(1) }
				\frac{1}{1-u^2}\, du \\
			&= \tanh^{-1}(u)\, \Bigr|_{\|p\|_T }
					^{\|q\|_T } \\
			&= \tanh^{-1} \bigl( \|q\|_T \bigr) - \tanh^{-1} \bigl( \|p\|_T \bigr).
\end{align*}
\end{proof}

The proof of Theorem~\ref{minimizerthm} now proceeds in four steps.  First, given a curve $\gamma$ from $p$ to $q$, we construct a new curve $\widetilde{\gamma}$ that is no longer than $\gamma$ and also passes through $m(\gamma)$.  Moreover while not fully monotonic itself, $\widetilde{\gamma}$ is the concatenation of two fully monotonic curves.  Second, a construction is performed to show that $\lambda_{p, q}$ is no longer than $\widetilde{\gamma}$.  Third, the lengths of all fully monotonic curves passing through $m(p, q)$ are shown to be the same.  Fourth, we show that no other curves have this length.

At first glance, the first and second steps may seem redundant with the fourth step.   We do need all four steps though, as the first and second steps show that L-shaped curves really are optimal, and the fourth step, which is both subtle and technical, shows that all optimal curves meet the given geometric criteria.

\begin{proof}[Proof of Theorem~\ref{minimizerthm}]

Step 1:  Define two new curves $\underline{\gamma}$ and $\overline{\gamma}$ by using the cumulative minimum and residual minimum coordinate-wise:
\[
\Bigl(\underline{\gamma}\Bigr)_i(t) = \underline{\gamma_i}(t)\ \mbox{and}\ 
	\Bigl(\overline{\gamma}\Bigr)_i(t) = \overline{\gamma_i}(t).
\]
By Proposition~\ref{minfunctionprop}, for each $i$, $\Bigl(\underline{\gamma}\Bigr)_i(1) = \Bigl(\overline{\gamma}\Bigr)_i(0)$, so the concatenation $\widetilde{\gamma} = \underline{\gamma} \ast \overline{\gamma}$ is well defined.  See Figure~\ref{adjustedcurvefig} for examples when $n = 2$.

\begin{figure}
\begin{picture}(360,310)
\put(0,10){
\put(0,160){
\includegraphics[scale = .65, clip = true, draft = false]{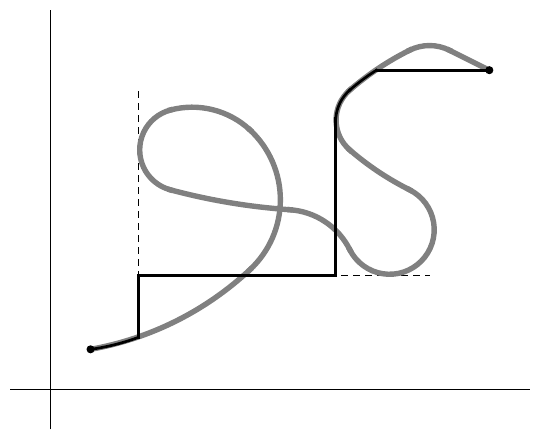}
}
\put(162,274){$p$}
\put(28,181){\footnotesize{$q = m(p, q) = m(\gamma)$}}
\put(145,224){$\gamma$}
\put(85,203){\footnotesize{$\underline{\gamma} \sim \widetilde{\gamma}$}}
\put(90,160){(a)}

\put(190,160){
\includegraphics[scale = .65, clip = true, draft = false]{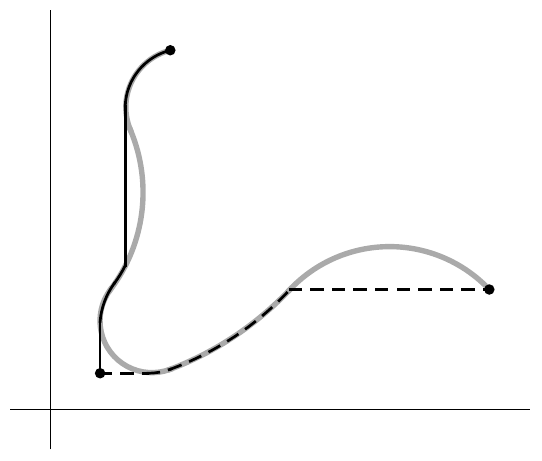}
}
\put(252,287){$p$}
\put(351,210){$q$}
\put(211,179){\footnotesize{$m(\gamma)$}}
\put(243,241){$\gamma$}
\put(222,240){$\underline{\gamma}$}
\put(315,204){$\overline{\gamma}$}
\put(270,160){(b)}

\put(0,20){
\includegraphics[scale = .65, clip = true, draft = false]{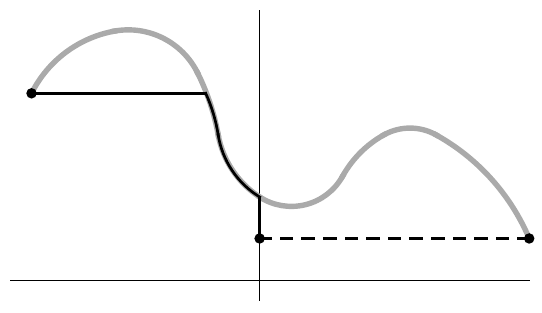}
}
\put(4,82){$p$}
\put(172,44){$q$}
\put(13,42){$m(p, q) = m(\gamma)$}
\put(128,83){$\gamma$}
\put(38,81){$\underline{\gamma}$}
\put(130,48){$\overline{\gamma}$}
\put(90,-5){(c)}

\put(185,0){
\includegraphics[scale = .65, clip = true, draft = false]{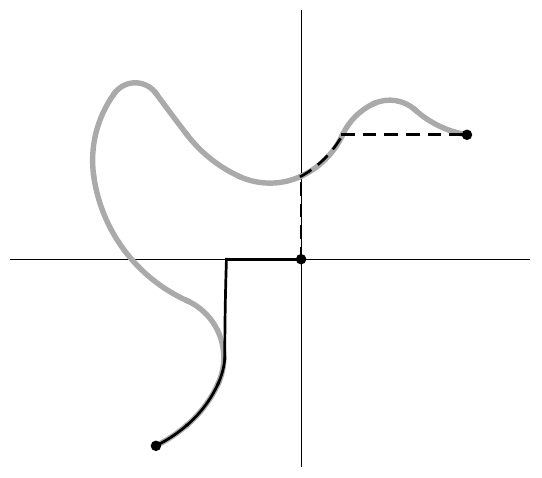}
}
\put(227,4){$p$}
\put(339,110){$q$}
\put(276,58){$m(p, q) = m(\gamma) = \bf{0}$}
\put(207,95){$\gamma$}
\put(262,52){$\underline{\gamma}$}
\put(311,98){$\overline{\gamma}$}
\put(270,-5){(d)}
}

\end{picture}
\caption{Various examples when $n = 2$ of curves $\gamma$ in gray, from $p$ to $q$, and the adjustments $\underline{\gamma}$, solid black and $\overline{\gamma}$, dashed black, which are concatenated to produce $\widetilde{\gamma}$.  In (a), $p$ and $q$ lie in the same quadrant with $p$ lying beyond $q$.   In this example, $\overline{\gamma}$ is stationary so $\widetilde{\gamma} \sim \underline{\gamma}$.  In (b), $p$ and $q$ lie in the same quadrant and neither point is beyond the other.  In (c) $p$ and $q$ lie in adjacent quadrants and it happens that $m(p, q) = m(\gamma)$ which need not always occur.  In $(d)$, $p$ and $q$ lie in opposite quadrants.  Again, $m(p, q) = m(\gamma)$ which turns out always to be true when $p$ and $q$ lie in opposite orthants.
} \label{adjustedcurvefig}
\end{figure}

Also by Proposition~\ref{minfunctionprop}, $\widetilde{\gamma}(0) = p$ and $\widetilde{\gamma}(1) = q$.  Moreover
\[
\left| \underline{\gamma_i}(t) \right| \leq |\gamma_i(t)|\ \mbox{and}\ \left| \overline{\gamma_i}(t) \right| \leq |\gamma_i(t)|
\]
for all $t \in [0, 1]$, so that
\[
\frac{1}{1 - \left| \left|\underline{\gamma}(t)\right|\right|_T ^2}
	\leq \frac{1}{1 -  \|\gamma(t)\|_T ^2}
\]
and
\[
\frac{1}{1 - \left| \left|\overline{\gamma}(t)\right| \right|_T ^2}
	\leq \frac{1}{1 -  \|\gamma(t)\|_T ^2}.
\]
Also by Proposition~\ref{minfunctionprop},
\[
|\underline{\gamma}_i'(t)| + |\overline{\gamma}_i'(t)| \leq |\gamma_i'(t)|
\]
whenever all three terms are defined.

These estimates imply
\begin{align*}
\len(\widetilde{\gamma}) &= \len(\underline{\gamma}) + \len(\overline{\gamma}) \\
				&= \int_0^1 \frac{\left|\left|  \underline{\gamma}'(t)\right| \right|_T}
								{1 -   \left|\left| \underline{\gamma}(t)\right|\right|_T^2} \, dt
					+ \int_0^1 \frac{\left| \left|\overline{\gamma}'(t)\right| \right|_T}
								{1 -   \left| \left| \overline{\gamma}(t)\right| \right|_T ^2} \, dt \\
				& \leq  \int_0^1 \frac{\left|\left| \underline{\gamma}'(t)\right| \right|_T}
								{1 -   \left| \left|\gamma(t)\right| \right|_T ^2} \, dt
					+ \int_0^1 \frac{\left|\left| \overline{\gamma}'(t)\right| \right|_T}
								{1 -  \left|\left| \gamma(t) \right| \right|_T ^2} \, dt\\
				& =  \int_0^1 \frac{\sum_{i = 1}^n \bigl(|\underline{\gamma}_i'(t)|
												+ |\overline{\gamma}_i'(t)|\bigr)}
								{1 -  \left| \left|\gamma(t) \right| \right|_T ^2} \, dt \\
				& \leq \int_0^1 \frac{\left| \left| \gamma'(t)\right| \right|_T}
								{1 -  \left| \left| \gamma(t) \right| \right|_T ^2} \, dt \\
				& = \len(\gamma).
\end{align*}

Step 2:  Consider the curves
\begin{align*}
\alpha &= \sigma_{p, m(p, q)} \ast \sigma_{m(p, q), m(\gamma)} \\
\beta &= \sigma_{m(\gamma), m(p, q)} \ast \sigma_{m(p, q), q} \\
\varphi &= \alpha \ast \beta.
\end{align*}
Note that $\alpha$ and $\beta$ are both fully monotonic since $m(p, q)$ lies beyond $m(\gamma)$.  Also, by Proposition~\ref{minfunctionprop}, $\underline{\gamma}$ and $\overline{\gamma}$ are both fully monotonic.  By Lemma~\ref{dmqbeyondplemma}, this implies that $\len(\alpha) = \len(\underline{\gamma})$ and $\len(\beta) = \len(\overline{\gamma})$ so $\len(\varphi) = \len(\widetilde{\gamma})$.  Hence,
\begin{align*}
\len(\lambda_{p, q}) &= \len(\sigma_{p, m(p, q)} \ast \sigma_{m(p, q), q}) \\
				&\leq \len(\sigma_{p, m(p, q)} \ast \sigma_{m(p, q), m(\gamma)}
							\ast \sigma_{m(\gamma), m(p, q)} \ast \sigma_{m(p, q), q}) \\
				&= \len(\varphi) \\
				&= \len(\widetilde{\gamma})
\end{align*}
Note that the inequality on the second line will be strict if $m(\gamma) \neq m(p, q)$.  This fact will be helpful in the fourth step.  See Figure~\ref{curveprogressionfig} for an example of the sequence of adjustments described here.

\begin{figure}
\begin{picture}(360,300)
\put(0,0){
\put(0,160){
\includegraphics[scale = .65, clip = true, draft = false]{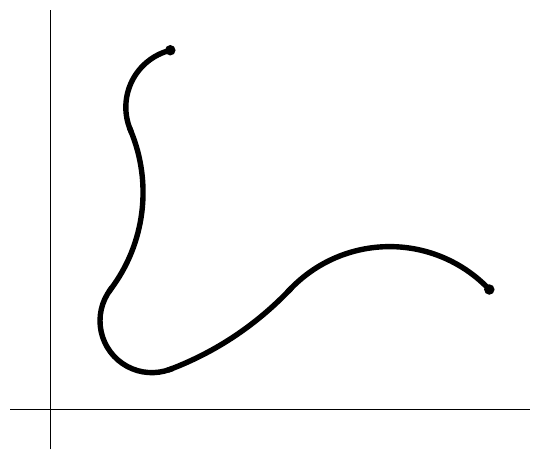}
}
\put(62,287){$p$}
\put(161,210){{$q$}}
\put(89,221){$\gamma$}
\put(175,240){$\longrightarrow$}

\put(190,160){
\includegraphics[scale = .65, clip = true, draft = false]{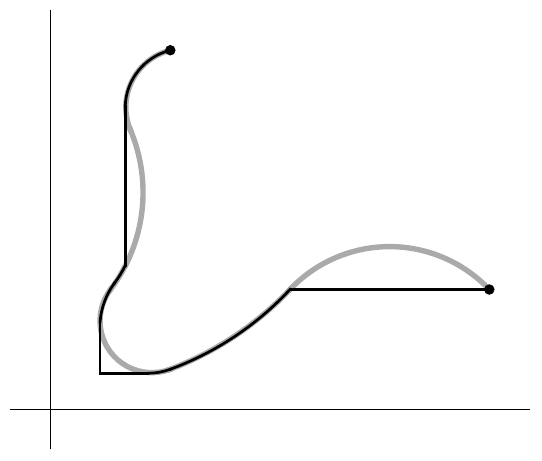}
}
\put(252,287){$p$}
\put(351,210){$q$}
\put(295,204){$\widetilde{\gamma}$}

\put(175,150){\rotatebox[origin=c]{225}{$-\!\!-\!\!\!\longrightarrow$}}

\put(0,0){
\includegraphics[scale = .65, clip = true, draft = false]{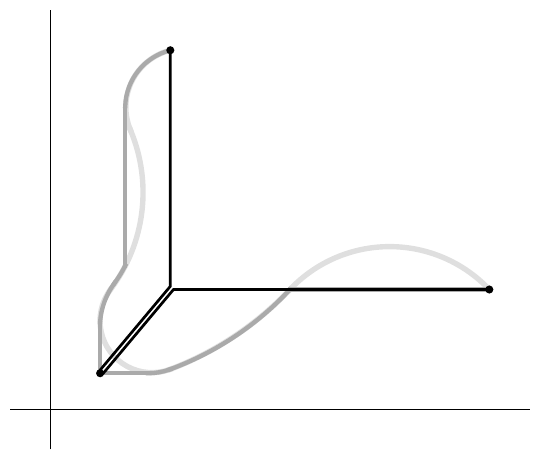}
}
\put(62,127){$p$}
\put(161,50){$q$}
\put(62,70){$\varphi$}
\put(21,19){\footnotesize{$m(\gamma)$}}

\put(175,80){$\longrightarrow$}

\put(190,0){
\includegraphics[scale = .65, clip = true, draft = false]{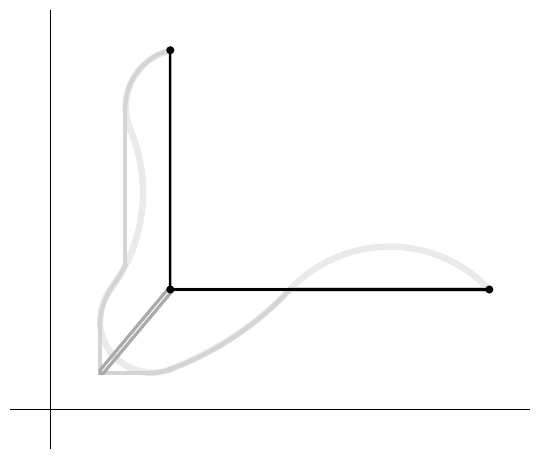}
}
\put(252,127){$p$}
\put(351,50){$q$}
\put(252,80){$\lambda_{p, q}$}
\put(245,45){\footnotesize{$m(p, q)$}}
}

\end{picture}
\caption{An example when $n = 2$ showing the sequence of steps adjusting the curve $\gamma$ through $\widetilde{\gamma}$ and $\varphi$ to $\lambda_{p, q}$.  In each step, the length of the new curve is no greater than that of the previous curve.  In the third step, the portions of $\varphi$ from $m(p, q)$ to $m(\gamma)$ and back are separated slightly for visual effect.} \label{curveprogressionfig}
\end{figure}

Step 3:  Let $\gamma: [0, 1] \rightarrow \hyp$ be any fully monotonic curve from $p$ to $q$ and passing through $m = m(p, q)$ at time $t^*$.  Let $\gamma^1$ and $\gamma^2$ be the restrictions of $\gamma$ to $[0, t^*]$ and $[t^*, 1]$ respectively so that  $\gamma \sim \gamma^1 \ast \gamma^2$.  Then by Lemma~\ref{dmqbeyondplemma}, since $p$ and $q$ each lie beyond $m$,
\begin{align*}
\len(\gamma) &= \len(\gamma^1) + \len(\gamma^2) \\
			&= \tanh^{-1} \bigl( \|p\|_T \bigr) - \tanh^{-1} \bigl( \|m\|_T \bigr).
					+ \left[ \tanh^{-1} \bigl( \|q\|_T \bigr) - \tanh^{-1} \bigl( \|m\|_T \bigr) \right] \\
			&= \tanh^{-1} \bigl( \|p\|_T \bigr) + \tanh^{-1} \bigl( \|q\|_T \bigr)
				- 2 \tanh^{-1} \bigl( \|m\|_T \bigr)
\end{align*}
As such, all fully monotonic curves from $p$ to $q$ and passing through $m$ have the same length.

Step 4:  It must be shown that if $\gamma$ is not fully monotonic or if $\gamma$ does not pass through $m(p, q)$ then $\len(\gamma) > \len(\lambda_{p, q})$.  Consider three cases:

Case 1:  Suppose $\gamma$ is not fully monotonic and in particular there exists a coordinate $i$ and an interval $[a, b]$ such that $|\gamma_i(a)| = |\gamma_i(b)|$ and for all $t \in (a, b)$, $|\gamma_i(t)| > |\gamma_i(a)|$.  Consider the curve $\check{\gamma}$  with components defined as follows:
\[
\check{\gamma}_j(x) =
\begin{cases}
\gamma_j(t) & \mbox{if} \ j \neq i\ \mbox{or\ if}\ j = i\ \mbox{and}\ t \notin [a, b], \\
\gamma_i(a) & \mbox{if} \ j = i\ \mbox{and}\ t \in [a, b].
\end{cases}
\]
Then $\check{\gamma}$ is a curve from $p$ to $q$ and, since $\gamma_i$ is absolutely continuous, there exists a set of positive measure in $[a, b]$ where $|\gamma_i'(t)| > 0$ so that
\[
0 = \int_a^b \frac{\left| \check{\gamma}_i'(t) \right|}{1 - \|\check{\gamma}(t)\|_T^2}\, dt
	< \int_a^b \frac{\left| \gamma_i'(t) \right|}{1 - \|\check{\gamma}(t)\|_T^2}\, dt
	\leq \int_a^b \frac{\left| \gamma_i'(t) \right|}{1 - \|\gamma(t)\|_T^2}\, dt.
\]
Hence, with the help of Steps 1 and 2 above,
\begin{align*}
\len(\lambda_{p, q}) &\leq \len(\check{\gamma}) \\
			&= \int_0^a  \frac{\left\| \gamma'(t) \right\|_T}{1 - \|\gamma(t)\|_T^2}\, dt
				+ \sum_{j = 1}^n
					\int_a^b  \frac{\left| \check{\gamma}_j'(t) \right|}{1 - \|\check{\gamma}(t)\|_T^2}\, dt
				+ \int_b^1  \frac{\left\| \gamma'(t) \right\|_T}{1 - \|\gamma(t)\|_T^2}\, dt \\
			&< \int_0^a  \frac{\left\| \gamma'(t) \right\|_T}{1 - \|\gamma(t)\|_T^2}\, dt
				+ \sum_{j = 1}^n
					\int_a^b  \frac{\left| \gamma_j'(t) \right|}{1 - \|\gamma(t)\|_T^2}\, dt
				\int_b^1  \frac{\left\| \gamma'(t) \right\|_T}{1 - \|\gamma(t)\|_T^2}\, dt \\
			&= \len(\gamma).
\end{align*}

Case 2:  Suppose $\gamma$ is not fully monotonic, but does not meet the criteria for Case~1.  Then $m(p, q)$ must lie beyond $m(\gamma)$ so the inequality in Step 2 above is strict and again $\len (\lambda_{p, q}) < \len(\gamma)$.

Case 3:  Suppose $\gamma$ is fully monotonic, but does not pass through $m(p, q)$.  Let $T_i = \{ t \in [0, 1]: \gamma_i(t) = m_i(p, q)\}$.  Since $\gamma$ is fully monotonic, each $T_i$ is an interval and since $\gamma$ does not pass through $m(p, q)$, $\bigcap_{i = 1}^n T_i = \emptyset$.  Hence, by Helly's theorem, there must be some pair of these intervals with an empty intersection.  Without loss of generality, suppose $T_1 \cap T_2 = \emptyset$.  Let $T_1 = [a, b]$ and $T_2 = [c, d]$, and again without loss of generality, suppose $b < c$.  Note that, since $\gamma_i$ is monotonic, if there is a time when $\gamma_i(t) = 0$, then $m(p, q) = 0$.  Hence, neither $\gamma_1$ nor $\gamma_2$ can be equal to zero on the interval $(b, c)$, so again without loss of generality, suppose that both $\gamma_1$ and $\gamma_2$ are positive on $(b, c)$.  Moreover, in this setting, $\gamma_1$ must be increasing and $\gamma_2$ must be decreasing.

Now let $\alpha = \underline{\gamma|_{[b, c]}}$, let $\beta = \overline{\gamma|_{[b, c]}}$ and consider the new curve
\[
\check{\gamma}(t) = 
\begin{cases}
\gamma(t) &\mbox{if}\  t \in [0, b) \\
\alpha \ast \beta(t) &\mbox{if}\  t \in [b, c] \\
\gamma(t) &\mbox{if}\  t \in (c, 1]. \\
\end{cases}
\]

For $t \in (b, c)$, since $|\alpha_1(t)| < |\gamma_1(t)|$, it follows that $\|\alpha(t) \|_T < \|\gamma(t) \|_T$.  Similarly, $\|\beta(t) \|_T < \| \gamma(t) \|_T$.  Focusing on the derivatives, $\alpha_1$ and $\beta_2$ are both constant, but since $\gamma_2$  must be non-constant somewhere on the interval $[b, c]$, there must be a set of positive measure where $\gamma_2'$ is non-zero.  Hence, on this set, since $|\alpha_2'(t)| \leq |\gamma_2'(t)|$,
\[
\frac{|\alpha_2'(t)|}{1 - \|\alpha(t) \|_T^2} < \frac{|\gamma_2'(t) |}{1 - \|\gamma(t) \|_T^2}
\]
and as a consequence
\[
\int_b^c \frac{|\alpha_2'(t)|}{1 - \|\alpha(t) \|_T^2}\, dt < \int_b^c \frac{|\gamma_2'(t) |}{1 - \|\gamma(t) \|_T^2}\, dt.
\]
Similarly,
\[
\int_b^c \frac{|\beta_1'(t) |}{1 - \|\beta(t) \|_T^2}\, dt < \int_b^c \frac{|\gamma_1'(t) |}{1 - \|\gamma(t) \|_T^2}\, dt.
\]
With these strict inequalities, and with the help of steps 1 and 2 above, we have
\begin{align*}
\len(\alpha &\ast \beta) \\
		&= \len(\alpha) + \len(\beta) \\
		&= \int_b^c \frac{\| \alpha'(t) \|_T}{1 - \| \alpha(t) \|_T^2}\, dt
			+ \int_b^c \frac{\| \beta'(t) \|_T}{1 - \| \beta(t) \|_T^2}\, dt \\
		&= \int_b^c \frac{|\alpha_1'(t)| + |\alpha_2'(t)| + \sum_{i=3}^n |\alpha_i'(t)|}{1 - \|\alpha(t)\|_T^2}\, dt
			+  \int_b^c \frac{|\beta_1'(t)| + |\beta_2'(t)| + \sum_{i=3}^n |\beta_i'(t)|}{1 - \|\beta(t)\|_T^2}\, dt \\
		& = \int_b^c \frac{0}{1 - \|\alpha(t)\|_T^2}\, dt + \int_b^c \frac{|\alpha_2'(t)|}{1 - \|\alpha(t)\|_T^2}\, dt
			+ \int_b^c \frac{\sum_{i = 3}^n |\alpha_i'(t)|}{1 - \|\alpha(t)\|_T^2}\, dt \\
		& \qquad \qquad	+ \int_b^c \frac{|\beta_1'(t)|}{1 - \|\beta(t)\|_T^2}\, dt
					+ \int_b^c \frac{0}{1 - \|\beta(t)\|_T^2}\, dt
					+ \int_b^c \frac{\sum_{i = 3}^n |\beta_i'(t)|}{1 - \|\beta(t)\|_T^2}\, dt \\
		& < \int_b^c \frac{|\gamma_1'(t) |}{1 - \|\gamma(t) \|_T^2}\, dt
			+ \int_b^c \frac{|\gamma_2'(t) |}{1 - \|\gamma(t) \|_T^2}\, dt \\
		& \qquad \qquad \qquad + \int_b^c \frac{\sum_{i = 3}^n |\alpha_i'(t)|}{1 - \|\alpha(t)\|_T^2}\, dt
				+ \int_b^c \frac{\sum_{i = 3}^n |\beta_i'(t)|}{1 - \|\beta(t)\|_T^2}\, dt \\
		& \leq \int_b^c \frac{|\gamma_1'(t) |}{1 - \|\gamma(t) \|_T^2}\, dt
			+ \int_b^c \frac{|\gamma_2'(t) |}{1 - \|\gamma(t) \|_T^2}\, dt \\
		& \qquad \qquad \qquad + \int_b^c \frac{\sum_{i = 3}^n |\alpha_i'(t)|}{1 - \|\gamma(t)\|_T^2}\, dt
				+ \int_b^c \frac{\sum_{i = 3}^n |\beta_i'(t)|}{1 - \|\gamma(t)\|_T^2}\, dt \\
		& = \int_b^c \frac{|\gamma_1'(t) |}{1 - \|\gamma(t) \|_T^2}\, dt
			+ \int_b^c \frac{|\gamma_2'(t) |}{1 - \|\gamma(t) \|_T^2}\, dt
			 + \int_b^c \frac{\sum_{i = 3}^n |\alpha_i'(t)| + |\beta_i'(t)|}{1 - \|\gamma(t)\|_T^2}\, dt \\
		& \leq \int_b^c \frac{|\gamma_1'(t) |}{1 - \|\gamma(t) \|_T^2}\, dt
			+ \int_b^c \frac{|\gamma_2'(t) |}{1 - \|\gamma(t) \|_T^2}\, dt
			 + \int_b^c \frac{\sum_{i = 3}^n |\gamma_i'(t)|}{1 - \|\gamma(t)\|_T^2}\, dt \\
		& = \int_b^c \frac{ \| \gamma'(t) \|_T}{1 - \| \gamma(t) \|_T^2}\, dt \\
		& = \len \left( \gamma|_{[b, c]} \right)
\end{align*}
and therefore
\begin{align*}
\len(\lambda_{p, q}) &\leq \len(\check{\gamma}) \\
				&= \len \left(\check{\gamma}|_{[0, b]} \right)
					+  \len \left(\check{\gamma}|_{[b, c]} \right)
					+ \len \left(\check{\gamma}|_{[c, 1]} \right) \\
				&=  \len \left({\gamma|_{[0, b]}} \right)
					+  \len \left({\alpha \ast \beta} \right)
					+  \len \left({\gamma|_{[c, 1]}} \right) \\
				&<  \len \left({\gamma|_{[0, b]}} \right)
					+  \len \left({\gamma|_{[b, c]}}  \right)
					+  \len \left({\gamma|_{[c, 1]}} \right) \\
				&= \len(\gamma).
\end{align*}

\end{proof}

\section{Distance and isometries} \label{distanceisomsec}

In this section, we use the length minimizers to define a distance function on the taxicab Poincar\'{e} ball
and determine the isometries of $\hyp$.

\subsection{Distance function for $\hyp$}

Here we prove Theorem~\ref{distthm}.

\medskip
\noindent \textbf{Theorem \ref{distthm}.}
\emph{
The distance function on $\hyp$ arising from the length functional $\len$ is
\begin{equation}  \label{BTdisteq}
d(p, q) = \tanh^{-1} \bigl( \|p\|_T \bigr) + \tanh^{-1} \bigl( \|q\|_T \bigr)
				- 2 \tanh^{-1} \bigl( \|m\|_T \bigr)
\end{equation}
where $m = m(p, q)$.
}
\medskip

\begin{proof}
By Theorem~\ref{minimizerthm}, $d(p, q) = \len(\lambda_{p, q})$ and Equation~\eqref{BTdisteq} is computed in the third step of the proof of  Theorem~\ref{minimizerthm}.
\end{proof}

Note that $d(p, \origin) = \tanh^{-1}\left(\|p\|_T \right)$, so Equation~\eqref{BTdisteq} can be rewritten as follows:
\begin{equation} \label{TPDdistrelationeq}
d(p, q) = d(p, \origin) + d(q, \origin) - 2 d(m, \origin).
\end{equation}

\subsection{Spheres}
The distance function for $\hyp$ could be used to characterize spheres in this space.  In taxicab space, taxicab spheres are the boundaries of orthoplexes, while in the usual Poincar\'{e} ball, spheres are Euclidean spheres, although the hyperbolic center does not coincide with the Euclidean center.  Unfortunately, most spheres in $\hyp$ are not taxicab spheres, and, while spheres centered at points other than the origin could be characterized, with explicit equations determined for those points that lie in the sphere, a detailed characterization is fairly cumbersome and not necessary for this paper.  That being said, spheres centered at the origin do provide some geometric insight and utility.

\begin{lemma} 
The spheres centered at the origin are orthoplexes.
\end{lemma}

\begin{proof}
The formula for a sphere of radius $r$ centered at the origin is $d(p, \origin) = \tanh^{-1}\! \left(\|p\|_T \right) = r$ which can be rewritten
\[
\|p\|_T = \tanh(r),
\]
which in turn is the formula of the taxicab sphere in $\mathbb{R}^n$ with radius $\tanh(r)$ centered at $\origin$, and hence an orthoplex.
\end{proof}

Let $r_p$ be the radius of the sphere centered at $\origin$ that contains $p$.  Note that, using this notation, Equation~\eqref{TPDdistrelationeq} becomes
\begin{equation} \label{TPDradiuseqn}
d(p, q) = r_p + r_q - 2 r_m.
\end{equation}

\subsection{Intervals}

We define the interval of $p$ and $q$, denoted $[p, q]$ as follows:
\[
[p, q] = \{x \in \hyp: d(p, q) = d(p, x) + d(x, q)\}.
\]

Note that, in general, intervals are isometry invariant:  if $\Psi$ is an isometry then $\Psi \bigl( [p, q] \bigr) = \bigl[ \Psi(p), \Psi(q) \bigr]$.  For $\hyp$ in particular, their geometry can be characterized explicitly.  Together, Theorems~\ref{minimizerthm} and \ref{distthm} imply the following lemma, the proof of which is left to the reader.

\begin{lemma} \label{intervalcharlemma}
Given two points $p$ and $q$ in $\hyp$, let $m = m(p, q)$.  Then
\[
[p, q] = R_{p, m} \cup R_{m, q}.
\]
\end{lemma}
See Figure~\ref{intervalfig} for some examples of these intervals when $n = 2$.  By way of comparison, intervals in $(\mathbb{R}^n, d_T)$ are orthotopes:  $[p, q]_T = R_{p, q}$, and intervals in the Poincar\'{e} Ball are arcs of circles perpendicular the hypersphere at infinity.

\begin{figure}
\begin{picture}(260,240)
\put(0,-10){
\includegraphics[scale = 1, clip = true, draft = false]{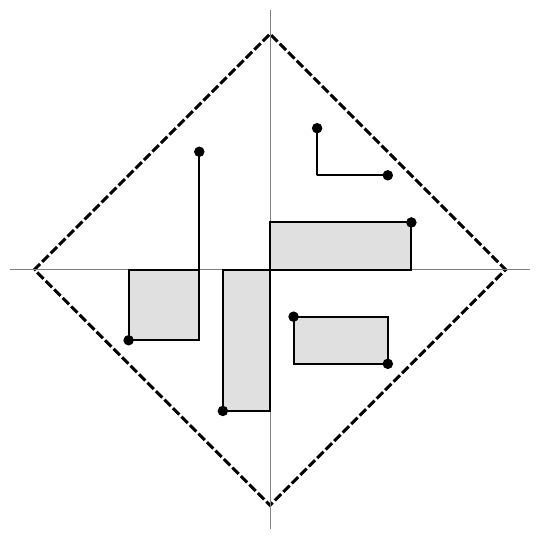}
}

\end{picture}
\caption{Various intervals in $\hyp^2$.
} \label{intervalfig}
\end{figure}

Lemma~\ref{intervalcharlemma} implies the following facts that will be useful in the proof of Theorem~\ref{isometrythm}:  First, in $\hyp^n$ if $p$ lies in a given $k$-orthant, the interval $[0, p]$ is a $k$-orthotope.  Second, for any point $p$ other than the origin, the set of all intervals $[p, q]$ includes sets that are not orthotopes.

\subsection{Isometries} \label{isometrysubsec}

Here we prove Theorem~\ref{isometrythm}:

\medskip
\noindent \textbf{Theorem \ref{isometrythm}.} \emph{
The isometry group for $\hyp^n$ is isomorphic to $H_n$.
}
\medskip

\begin{proof}
The hyperoctahedral group can be thought of as acting on $\hyp$, generated by permutations of the coordinates and reflections across coordinate hyperplanes containing the origin.  The fact that these transformations are isometries is left to the reader, as is the fact that they generate the group $H_n$.

In the other direction, we proceed in five steps.  First, by Lemma~\ref{intervalcharlemma}, since the origin is the only point with the property that all intervals $[0, p]$ are othotopes, any isometry $\Psi$ must map the origin to itself.

Second, if $p$ lies in a $k$-orthant, $\Psi(p)$ must as well since $\Psi([0, p]) = [0, \Psi(p)]$ must both be $k$-orthotopes.

Third, suppose $\Psi$ maps a $1$-orthant to itself.  Then $\Psi$ is the identity map on that orthant.  This follows directly from the fact that distance is strictly monotonic.

Fourth, coordinate axes must map to coordinate axes.  (They cannot be bent.)  To see this, by way of contradiction, suppose, without loss of generality, that $\Psi$ fixes all positive axes and that $\Psi$ maps the negative $x_2$-axis to the negative $x_1$-axis.  Consider the curve $\sigma_{p,q}$ where $p = (1, 0, 0, \ldots, 0)$ and $q = (0, -1, 0, \ldots, 0)$.  Then $\Psi(p) = p$ and $\Psi(q) = -p$.  Note that for $t \in (0, 1)$, the intervals $[0, \sigma_{p, q}(t)]$ are all 2-orthotopes., and hence $[0, \Psi(\sigma_{p, q}(t))]$ must all be 2-orthotopes.  However, every curve from $p$ to $-p$ must either pass through a coordinate axis, creating a 1-orthotope at that time, or it must leave the coordinate planes passing through the origin, creating $k$-orthotopes where $k \geq 3$.

Fifth, suppose $\Psi$ is the identity on all $1$-orthants.  Then $\Psi$ is the identity on all of $\hyp^n$.  To see this, let $\pi_i: \hyp^n \rightarrow \hyp^n$ be projection onto the $i^{\mathrm{th}}$ coordinate.  For any point $p$, let $q = \Psi(p)$.  Then $[0, p]$ and $\Psi\bigl([0, p]\bigr) = [0, q]$ are orthotopes which contain $[0, \pi_i(p)]$ and $[0, \pi_i(q)]$ respectively for all $i \in \{1, \ldots, n\}$.  But then
\[
\bigl[0, \pi_i(p)\bigr] = \Psi\Bigl(\bigl[0, \pi_i(p) \bigr] \Bigr)
				= \bigl[0, \pi_i(q) \bigr]
\]
and so $\pi_i(p) = \pi_i(q)$ which in turn implies that $p = q$.

\end{proof}

\section{Hyperbolicity and medians} \label{hyperbolicitysec}

Here, we determine the extent to which $\hyp$ is hyperbolic and explore its median structure.

\subsection{Gromov hyperbolicity}

Since we are not working with curvature, we instead work with Gromov hyperbolicity.

\medskip
\noindent \textbf{Theorem \ref{gromovhyperbolicthm}.} \emph{
$\hyp$ is $\ln(3)$-hyperbolic and this value of $\delta$ is optimal.
}
\medskip

Before proving Theorem~\ref{gromovhyperbolicthm} we introduce some notation and structure.

In \cite{CDP}, it is shown that if Inequality~\eqref{gromovineq} holds for a single base point $w_0$ and $\delta = \delta_0$, then it holds for all $w$ and at most $\delta = 2 \delta_0$.  With this in mind, let $w_0 = \origin$ and note that by Equation~\eqref{TPDradiuseqn}
\begin{align*}
G(p, q; \origin) &= \frac{1}{2} \left[ d(p, \origin) + d(q, \origin) - d(p, q) \right] \\
		&= \frac{1}{2} \left[ r_p + r_q - \left(r_p + r_q - 2 r_{m(p, q)}\right) \right] \\
		&= r_{m(p, q)}.
\end{align*}
Our goal will then be to show that for any three points $p$, $q$, and $s$
\[
G(p, q; \origin) \geq \min \left\{ G(p, s; \origin), G(q, s; \origin) \right\} - \frac{1}{2} \ln(3)
\]
which, using the formula above and rearranging, can be written
\[
\min \left\{ r_{m(p, s)}, r_{m(q, s)} \right\} - r_{m(p, q)} \leq \frac{1}{2} \ln(3).
\]
As such, letting
\[
h(p, q, s) = \min \left\{ r_{m(p, s)}, r_{m(q, s)} \right\} - r_{m(p, q)}
\]
we want to show that
\[
\sup_{p, q, s \in \hyp^n} h(p, q, s) = \frac{1}{2} \ln(3).
\]

To help with some of the technical aspects in the proof, given a set $S$ of numbers, we say a number in $S$ has a unique sign if it is nonzero and its sign differs from the signs of all the other numbers in the set.  Then, given three points $p$, $q$, and $s$ in $\hyp^n$, we define the adjusted points $p^*$, $q^*$, and $s^*$ as follows.  For $i \in \{1, \ldots, n\}$, consider the three-element set $C_i =  \{p_i, q_i, s_i\}$. If a number in this set has a unique sign, adjust the corresponding point by replacing the $i$th coordinate with zero.  Otherwise, if all nonzero values in $C_i$ have the same sign, make no adjustment.

Note that, since $C_i$ is a three-element set, either at least one element has a unique sign or all nonzero elements have the same sign.  Using this adjustment technique, we have the following:

\begin{lemma} \label{hadjustlemma}
Given three points $p$, $q$, and $s$ in $\hyp^n$, let $p^*$, $q^*$ and $s^*$ be the adjusted points by the method described above.  Then $m(p^*, q^*) = m(p, q)$, $m(p^*, s^*) = m(p, s)$, and $m(q^*, s^*) = m(q, s)$, and hence
\[
h(p^*, q^*, s^* ) = h(p, q, s).
\]
\end{lemma}

\begin{proof}
Recall that for any two points $x$ and $y$ in $\hyp^n$, and any index $i \in \{1, \ldots, n\}$ if $x_i$ and $y_i$ do not have the same sign, then $m_i(x, y) = 0$.  As such, if $x_i$ or $y_i$ is replaced by zero, $m(x, y)$ is not changed.  Since, among $p_i$, $q_i$, and $s_i$, a value is only adjusted to zero if it has a unique sign, this adjustment does not affect the minimal points.

The fact that $h$ is unchanged by the adjustment follows immediately from the fact that the minimal points are unchaged.
\end{proof}

We are now ready to prove Theorem~\ref{gromovhyperbolicthm}, which comprises three main parts.  First we show that for $n \geq 3$, the dimension of the space can be reduced without increasing the supremum for $h$.  Second, we show that when $n = 2$, the supremum for $h$ is $\frac{1}{2} \ln(3)$.  As mentioned above, if a supremum $\delta_0$ can be found for $h$, then Inequality~\eqref{gromovineq} holds for all $w$ and at most $\delta = 2 \delta_0$ so these first two steps show that $\delta \leq \ln(3)$ in all dimensions.  The final step is to show that $\delta = \ln(3)$ is also necessary in all dimensions.

\begin{proof}[proof of Theorem~\ref{gromovhyperbolicthm}]
By Lemma~\ref{hadjustlemma}, the points $p$, $q$, and $s$, can be adjusted to lie in the same orthant without altering the value of $h$.  As such, without loss of generality, suppose all three points lie in the all-nonnegative orthant.

Noting that $h$ is symmetric in its first two arguments, consider the first three coordinates of $p$ and $q$.  For one of these points, two of the coordinates are no greater than the corresponding coordinates of the other point, so without loss of generality, suppose $p_1 \leq q_1$ and $p_2 \leq q_2$.  Now apply the following transformation:  let $\widehat{p} = (p_1 + p_2, p_3, \ldots, p_n)$, $\widehat{q} = (q_1 + q_2, q_3, \ldots, q_n)$, and $\widehat{s} = (s_1 + s_2, s_3, \ldots, s_n)$.  Note that these points lie in $\hyp^{n-1}$.  Also,
\begin{align*}
r_{m(\widehat{p}, \widehat{q})} &= \tanh^{-1}(|m_1(\widehat{p}, \widehat{q})|
								+ |m_2(\widehat{p}, \widehat{q})|
								+ |m_3(\widehat{p}, \widehat{q})|
								+ \cdots + |m_n(\widehat{p}, \widehat{q})|) \\
						&= \tanh^{-1}(p_1 + p_2 + |m_3(p, q)| + \cdots + |m_n(p, q)|) \\
						&= \tanh^{-1}(|m_1(p, q)| + |m_2(p, q)| + |m_3(p, q)| + \cdots + |m_n(p, q)|) \\
						&= r_{m(p, q)}
\end{align*}
while
\begin{align*}
r_{m(\widehat{p}, \widehat{s})} &= \tanh^{-1}(|m_1(\widehat{p}, \widehat{s})|
								+ |m_2(\widehat{p}, \widehat{s})|
								+ |m_3(\widehat{p}, \widehat{s})|
								+ \cdots + |m_n(\widehat{p}, \widehat{s})|) \\
					&= \tanh^{-1}(\min \{p_1 + p_2, s_1 + s_2\} + |m_3(p, s)| + \cdots + |m_n(p, s)|) \\
					&\geq \tanh^{-1}(\min \{p_1, s_1\} + \min \{p_2, s_2\}
							+ |m_3(p, s)| + \cdots + |m_n(p, s)|) \\
					&= \tanh^{-1}(|m_1(p, s)| + |m_2(p, s)| + |m_3(p, s)| + \cdots + |m_n(p, s)|) \\
					&= r_{m(p, s)}
\end{align*}
and similarly
\[
r_{m(\widehat{q}, \widehat{s})} \geq  r_{m(q, s)}.
\]
Together, these imply that $h(\widehat{p}, \widehat{q}, \widehat{s}) \geq h(p, q, s)$, and this implies that for $n \geq 3$
\[
\sup_{p, q, s \in \hyp^n} h(p, q, s) \leq \sup_{p, q, s \in \hyp^{n-1}} h(p, q, s).
\]

The next step is to show that in dimension two, $\sup_{p, q, s \in \hyp^2} h(p, q, s) = \frac{1}{2}\ln(3)$.  By Lemma~\ref{hadjustlemma} the points can be adjusted to lie in the same quadrant without affecting the value of $h$, so without loss of generality, suppose all three points lie in the first quadrant.

Recalling again that $h$ is symmetric in the first two arguments, consider three cases.  First, if $s$ does not lie beyond $m(p, q)$, then $h(p, q, s) \leq 0$.  To see this, without loss of generality, suppose $s_1 < \min \{p_1, q_1\}$ so
\[
m_1(p, s) = m_1(q, s) = s_1.
\]
Then
\begin{align*}
m(p, s) &= (s_1, \min \{p_2, s_2\}), \\
m(q, s) &= (s_1, \min \{q_2, s_2\}), \\
m(p, q) &= (\min \{p_1, q_1\}, \min \{p_2, q_2\}).
\end{align*}
Note also that
\[
\min \{p_2, q_2\} \geq \min \Big\{ \min \{p_2, s_2\}, \min \{q_2, s_2\} \Big\}.
\]
Therefore $m(p, q)$ must lie beyond at least one of $m(p, s)$ and $m(q, s)$ and this implies that
\[
h(p, q, s) = \min \{r_{m(p, s)}, r_{m(q, s)} \} - r_{m(p, q)} \leq 0.
\]

For the second case, suppose $s$ lies beyond $m(p, q)$ and, without loss of generality, $p$ lies beyond $q$.  In this case, $m(p, q) = m(q, s) = q$ and $m(p, s)$ lies beyond $q$.  As such,
\[
h(p, q, s) = \min \{r_{m(p, s)}, r_{m(q, s)} \} - r_{m(p, q)} = 0.
\]

Third, suppose $s$ lies beyond $m(p, q)$ and neither $p$ nor $q$ lie beyond the other.  Suppose, without loss of generality, that $p_1 < q_1$ and $p_2 > q_2$.  Then $s_1 > p_1$  so $m(p, s) = (p_1, \min \{p_2, s_2\})$ and $(p_1, s_2)$ lies beyond this point.  Similarly, $s_2 > q_2$ so $m(q, s) = (\min \{q_1, s_1\}, q_2)$ and $(s_1, q_2)$ lies beyond this point.  See Figure~\ref{deltahyponequadfig} for an example of this scenario.  With these relationships, since $s_1 + s_2 < 1$
\begin{align*}
h(p, q, s) &= \min \{r_{m(p, s)}, r_{m(q, s)}\} - r_{m(p, q)} \\
		&\leq \min \{r_{(p_1, s_2)}, r_{(s_1, q_2)}\} - r_{(p_1, q_2)} \\
		&= \min \{\tanh^{-1}(p_1 + s_2), \tanh^{-1}(s_1 + q_2) \} - \tanh^{-1}(p_1 + q_2) \\
		& \leq \min \{\tanh^{-1}(p_1 + 1 - s_1), \tanh^{-1}(s_1 + q_2) \} - \tanh^{-1}(p_1 + q_2).
\end{align*}
Fixing $p_1$ and $q_2$, since the inverse hyperbolic tangent is monotonic, the first term will be maximized when $p_1 + 1 - s_1 = s_1 + q_2$ so that $s_1 = \frac{1 + p_1 - q_2}{2}$.  Hence
\begin{align*}
h(p, q, s) &\leq \tanh^{-1}\left(\frac{1 + p_1 + q_2}{2} \right) - \tanh^{-1}(p_1 + q_2) \\
		&= \tanh^{-1} \left( \frac{1 - (p_1 + q_2)}{2 - (1 + p_1 + q_2)(p_1 + q_2)} \right) \\
		&= \tanh^{-1} \left(\frac{1}{2 + (p_1 + q_2)} \right)
\end{align*}
where the last simplification is allowed since $p_1 + q_2 < 1$. Since $0 \leq p_1 + q_2 < 1$, this is maximized when $p_1 + q_2 = 0$ implying that $\sup_{p, q, s} h(p, q, s) \leq \tanh^{-1} \left( \frac{1}{2} \right) = \frac{1}{2}\ln(3)$.

\begin{figure}
\begin{picture}(240,200)
\put(10,0){
\includegraphics[scale = .8, clip = true, draft = false]{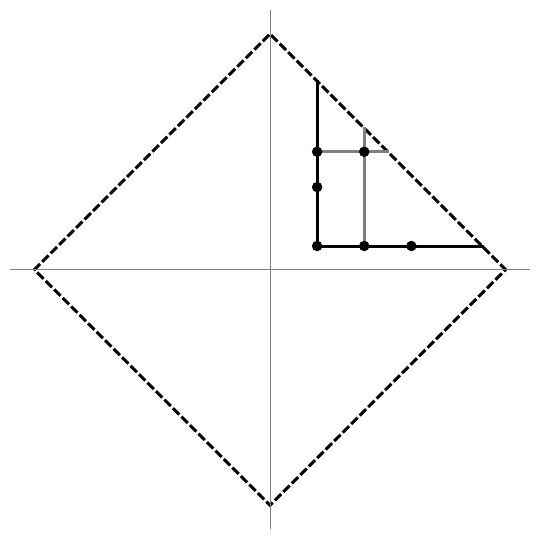}
}
\put(82,131){$p = m(p, s)$}
\put(104,105){$m(p, q)$}
\put(173,105){$q$}
\put(155,142){$s$}
\put(118,90){$(s_1, q_2) = m(q, s)$}
\put(154,97){\vector(0,1){13}}

\put(100,150){$(p_1, s_2)$}
\end{picture}
\caption{The points of interest in the third case of the second step of the proof of $\delta$-hyperbolicity.  In this example, $p$ and $s$ are positioned such that $m(p, s) = p$, and $(p_1, s_2)$ lies beyond this point.  Meanwhile, neither $q$ nor $s$ lie beyond the other and $m(q, s) = (s_1, q_2)$.
} \label{deltahyponequadfig}
\end{figure}

The final step is to show that in fact $\delta = \ln(3)$ is needed when the fourth point is allowed to vary.   Let $s = (t, t)$, $p = (-t, t)$, $q = (t, -t)$, and $w = (-t, -t)$.  Then the various minimal points here are $(\pm t, 0)$ and $(0, \pm t)$ and so by direct calculation,
\[
\min \{G(p, s; w), G(q, s; w) \} - G(p, q; w) = 2 \tanh^{-1}(t).
\]
Letting $t$ approach $\frac{1}{2}$ shows that
\[
\sup_{p, q, s, w \in \hyp} \bigl[ \min \{G(p, s; w), G(q, s; w) \} - G(p, q; w) \bigr] = 2 \tanh^{-1} \left(\frac{1}{2} \right) = \ln(3).
\]
See Figure~\ref{deltahypfactoroftwofig}.

\begin{figure}
\begin{picture}(240,200)
\put(10,0){
\includegraphics[scale = .8, clip = true, draft = false]{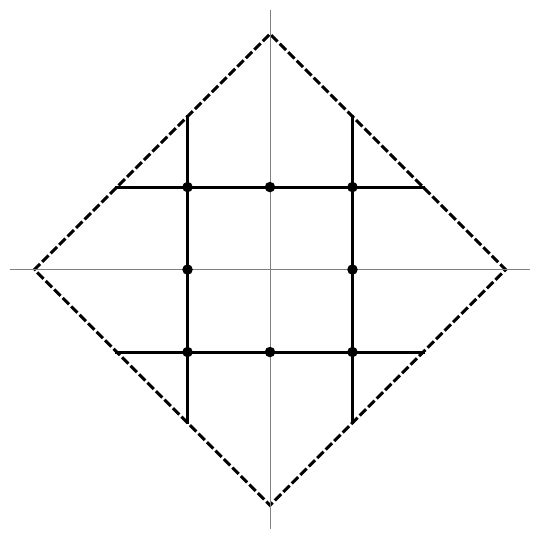}
}
\put(78,129){$p$}
\put(102,141){$m(p, s)$}
\put(50,95){$m(p, w)$}
\put(142,65){$q$}
\put(151,138){$s$}
\put(76,65){$w$}
\put(152,108){$m(q, s)$}
\put(102,63){$m(p, s)$}
\end{picture}
\caption{The points of interest in the third step of the proof of $\delta$-hyperbolicity.} \label{deltahypfactoroftwofig}
\end{figure}

This construction is shown for $n = 2$, but since $\hyp^2$ is naturally contained in $\hyp^n$ for all $n \geq 2$, $\delta = \ln(3)$ is necessary for all dimensions.

\end{proof}

\subsection{Medians}

Given three points $x$, $y$, and $z$, we define the median to be
\[
\mu(x, y, z) = [x, y] \cap [y, z] \cap [z, x]
\]
if it exists and is unique.

Before providing a characterization of those triples of points in $\hyp$ that have medians, we establish some supporting facts.

Given three points $x$, $y$, and $z$ in $\hyp$, define a new point $\nu(x, y, z)$ coordinate-wise by
\[
\nu_i(x, y, z) = \mbox{median}\{x_i, y_i, z_i\}
\]
where on the right, this is the numerical or statistical median.   Taxicab space $(\mathbb{R}^n, d_T)$ is a median space and $\mu(x, y, z) = \nu(x, y, z)$.  As we shall see, not all triples of points in $\hyp$ support a median, but if a median does exist, then, like taxicab space, it must be this coordinate-wise median point $\nu$.

\begin{lemma} \label{statmedianlemma}
If $p$, $q$, and $s$ in $\hyp$ support a median $\mu$, it must be the point $\nu$.
\end{lemma}

\begin{proof}
By Theorem~\ref{minimizerthm}, if a point $x$ lies in $[p, q]$, then it must lie in a fully monotonic curve from $p$ to $q$ and hence, for each index $i \in \{1, \ldots, n\}$, $\min\{p_i, q_i\} \leq x_i \leq \max\{p_i, q_i\}$.
Let $x \in [p, q] \cap [q, s] \cap [s, p]$ and fix $i$.  Without loss of generality, suppose $p_i \leq q_i \leq s_i$.  Then $p_i \leq x_i \leq q_i$ and $q_i \leq x_i \leq s_i$.  But this implies $x_i = q_i$.
\end{proof}

Note that this lemma shows that if $]p, q] \cap [q, s] \cap [s, p]$ is nonempty, it consists of a single point, so the uniqueness part of the requirement to be a median is satisfied.

The following technical lemma establishes relationships among the coordinates of minimal points defined by a given triple of points in $\hyp$.

\begin{lemma} \label{techmedianlemma}
Given three points $p$, $q$, and $s$ in $\hyp$ and an index $i \in \{1, \ldots, n\}$, consider the numbers $m_i(p, q)$, $m_i(p, s)$, and $m_i(q, s)$.  Then
\begin{itemize}
\item at least two of these numbers are equal;
\item if they are not all equal, the unique value is equal to $\nu_i(p, q, s)$.
\item Suppose $m_i(p, s) = m_i(q, s)$.  Then $m_i(m(p, q), s) = m_i(p, s)$.
\end{itemize}
\end{lemma}

\begin{proof}

For the first two statements, consider the triple $\{p_i, q_i, s_i\}$.  Observe that if there is no unique sign, then for two of the minimal points, the $i^{\mathrm{th}}$ coordinate is $\min\{p_i, q_i, s_i\}$, while if some number in $\{p_i, q_i, s_i\}$ has a unique sign, then two of the minimal points have an $i^{\mathrm{th}}$ coordinate of 0.  In both cases, the third minimal point is the middle value among the $\{p_i, q_i, s_i\}$.

For the third statement, consider two possibilities.  If no number in $\{p_i, q_i, s_i\}$ has a unique sign then since $m_i(p, s) = m_i(q, s)$, it follows that $|s_i| \leq |p_i|$ and $|s_i| \leq |q_i|$.  This implies both that $m_i(p, s) = m_i(q, s) = s_i$ and that $|s_i| \leq |m_i(p, q)|$ so $m_i(m(p, q), s) = s_i$. Hence
$m_i(p, s) = m_i(m(p, q), s)$.

On the other hand, if some number has a unique sign among $\{p_i, q_i, s_i\}$ then in fact $s_i$ has the unique sign, or $s_i = 0$.  In either case, $m_i(p, s) = m_i(q, s) = 0$, but also $m_i(m(p, q), s) = 0$.

Together, these two cases imply that $m_i(p, s) = m_i(m(p, q), s)$
\end{proof}

With this lemma in hand, we now have:

\begin{thm} \label{mediancharacterizationthm}
Given three points $p$, $q$, and $s$ in $\hyp$, if $m(p, s) = m(q, s)$, then $\{p, q, s\}$ supports a median and $\mu(p, q, s) = m(p, q)$.  Conversely, if the three minimal points defined by $p$, $q$, and $s$ are distinct, then $\{p, q, s\}$ does not support a median.
\end{thm}

\begin{proof}

Suppose $m(p, s) = m(q, s)$ and let $m = m(p, q)$.  Note that $m$ always lies in $[p, q]$ since, by Theorem~\ref{minimizerthm}, every geodesic connecting $p$ and $q$ must pass through $m$.

To confirm that $m \in [p, s]$, note that $d(p, m) = r_p - r_m$ while $d(p, s) = r_p + r_s - 2 r_{m(p, s)}$ so
\begin{align*}
d(p, m) + d(m, s) &= r_p - r_m + r_m + r_s - 2r_m(m, s) \\
			&= r_p + r_s - 2r_m(p, s) \\
			&= d(p, s). \\
\end{align*}
Hence $m \in [p, s]$.  Similarly, $m \in [q, s]$.

On the other hand, suppose $m$, $m(p, s)$, and $m(q, s)$ are distinct, but $p$, $q$, and $s$ support a median $\mu$ nonetheless.  By the first statement in Lemma~\ref{techmedianlemma}, for each index $i \in \{1, \ldots, n\}$ at least two of the numbers $m_i$, $m_i(p, s)$, and $m_i(q, s)$ are equal.  Since the three minimal points are distinct, it must then be the case that without loss of generality,
\[
\mu_1 = m_1(q, s) \neq m_1(p, q) = m_1(p, s)
\]
while
\[
\mu_2 = m_2(p, s) \neq m_2(p, q) = m_2(q, s).
\]

If for a given $i$, no number in $\{p_i, q_i, s_i\}$ has a unique sign, without loss of generality, suppose they are all nonnegative.  Similarly, if there is a number with a unique sign, without loss of generality, suppose the unique sign is negative.  Then for both $i$, $\mu_i \geq 0$ and so by Lemma~\ref{statmedianlemma} the third statement in Lemma~\ref{techmedianlemma},
\[
\mu_1 = m_1(q, s) > m_1(p, q) = m_1(p, s)
\]
while
\[
\mu_2 = m_2(p, s) > m_2(p, q) = m_2(q, s)
\]
so in the first two coordinates, $\mu$ lies beyond $m(p, q)$ and $p_1 < \mu_1 \leq q_1$ and $q_2 < \mu_2 \leq p_2$.

Hence, there exists a fully monotonic curve $\gamma$ connecting $p$, $q$, $m$, and $\mu$, and any such curve is increasing in the first coordinate and decreasing in the second.  But $m_1 < \mu_1$ and $m_2 < \mu_2$ which implies $\gamma$ is increasing in both coordinates or decreasing in both coordinates, and this results in a contradiction.

\end{proof}

With Theorem~\ref{mediancharacterizationthm} providing a characterization of medians in $\hyp$, examples of triples of points which do or do not support medians can more easily be found.  For example, for all $p$ and $q$ in $\hyp$, $\{p, q, \origin\}$ supports a median since $m(p, \origin) = m(q, \origin) = \origin$.  Additionally, in $\hyp^2$, if $p$ and $q$ are in adjacent quadrants and $s$ lies on the coordinate axis separating these quadrants, then $\{p, q, s\}$ supports a median since for at least one index $i$, $m_i(p, q)= m_i(p, s) = m_i(q, s) = 0$.  On the other hand, if $p$, $q$, and $s$ all lie in the same quadrant, neither $p$ nor $q$ lies beyond the other, and $s$ lies strictly beyond $m(p, q)$, then $\{p, q, s\}$ does not support a median, nor is a median supported if $p$, $q$, and $s$ each lie in their own quadrant.  See Figure~\ref{medianfig} for illustrations of these examples.  Finally, since $\hyp^2 \subset \hyp^n$, similar examples exist in higher dimensions as well.  Together, these examples imply

\begin{cor} \label{medspacecor}
The space $\hyp$ is not a median space.
\end{cor}

\begin{figure}
\begin{picture}(360,360)
\put(-10,180){
\includegraphics[scale = .7, clip = true, draft = false]{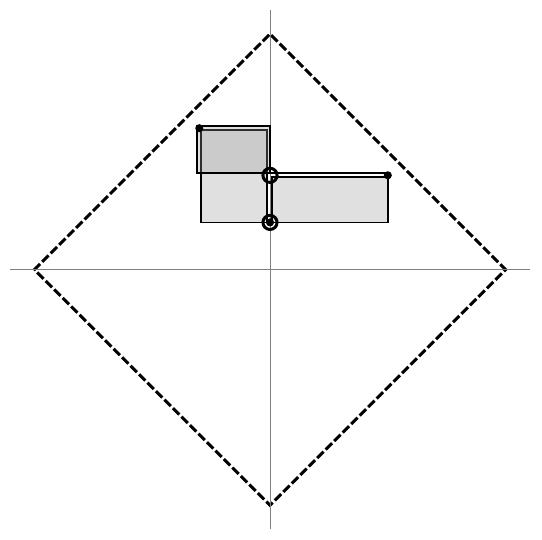}
}
\put(180,180){
\includegraphics[scale = .7, clip = true, draft = false]{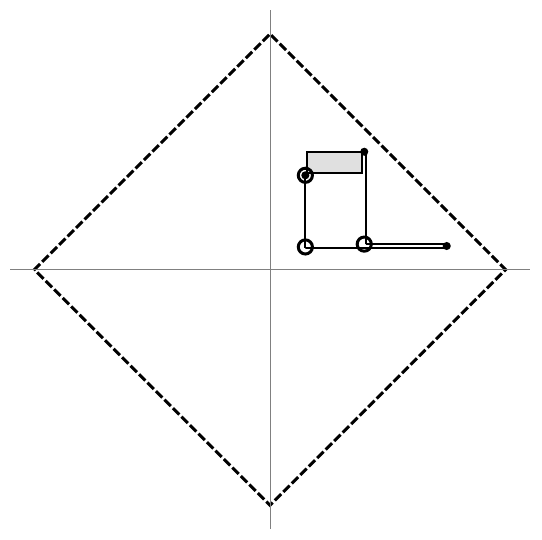}
}
\put(-10,0){
\includegraphics[scale = .7, clip = true, draft = false]{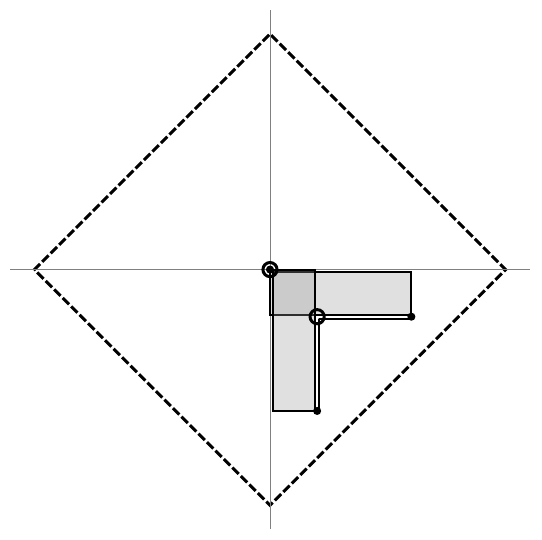}
}
\put(180,0){
\includegraphics[scale = .7, clip = true, draft = false]{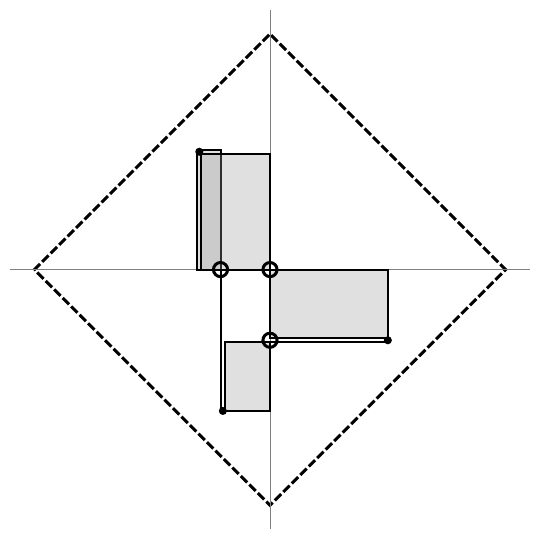}
}

\end{picture}
\caption{Triples of points that support or do not support a median.  The points in each triple are identified as black dots and their minimal points are identified with circles.  The intervals for each pair of points are also shown, the edges of which have been offset slightly from their actual positions for clarity.  On the left, in each triple, two of the minimal points coincide, and the third minimal point, which lies beyond the other two, is the median.  On the right, each triple produces three distinct minimal points and there is no median.}   \label{medianfig}
\end{figure}

\section{Final thoughts and next steps} \label{finalsec}

We finish with some ideas for future exploration related to our taxicab Poincar\'{e} ball and taxicab hyperbolic geometry more generally.

\subsection{Hyperbolic models}

In our attempt to modify the Poincar\'{e} ball model, the weight $\frac{1}{1-\|x\|_T^2}$ used to adjust the norm was used because of its algebraic similarity to the weight function used for the regular Poincar\'{e} ball.  It seems we were lucky that our choice was amenable to fairly straightforward analysis and resulted in clean results.  That being said, other weight functions may be worth considering.

Also, the Poincar\'{e} ball is one of many models for hyperbolic space.  It would be interesting to look at other models from the taxicab perspective presented here.  The upper half-space model would be the next natural choice.  We expect that identifying length minimizing curves in this setting would require analysis comparable to what was done for $\hyp$.  Interestingly, the isometry group for an upper half-space model would be infinite, allowing at least for translations parallel to the hyperplane boundary.

Alternatively, developing a taxicab hyperbolic space could be attempted from the perspective of the isometry group.  The isometry group for hyperbolic space is isomorphic to $O^+(n,1)$.  It would be interesting to search for a similar group that acts transitively on some set, has point stabilizers that are isomorphic to the hyperoctahedral group, thus exhibiting a local taxicab structure, and where the resulting space is hyperbolic in some sense.

Finally, the structure of the Poincar\'{e} ball is closely related to the geometric transformation of inversion.  Some initial attempts by the authors to develop a comparable transformation in the taxicab setting were unsuccessful, but more work is warranted in this direction.


\bibliographystyle{amsalpha}
\bibliography{taxicab-HG.bib}

\end{document}